\documentclass[a4paper, 11pt, final]{article}

\usepackage[all]{xy}
\usepackage{style}
 \usepackage{amsthm}
\usepackage{dsfont}

 \usepackage{geometry}

 \title{Chow groups of smooth varieties fibred by quadrics}

\author{Charles Vial}
 \date{}

\begin{document}

\maketitle

\begin{abstract} Let $f : X \r B$ be a proper flat dominant morphism
  between two smooth quasi-projective complex varieties $X$ and $B$.
  Assume that there exists an integer $l$ such that all closed fibres
  $X_b$ of $f$ satisfy $CH_0(X_b) = CH_1(X_b) = \ldots = CH_l(X_b) =
  \Q$. Then we prove an analogue of the projective bundle formula for
  $CH_i(X)$ for $i \leq l$. When $B$ is a surface, $X$ is projective
  and $l = \lfloor \frac{\dim X - 3}{2} \rfloor$, this makes it
  possible to construct a Chow-K\"unneth decomposition for $X$ that
  satisfies Murre's conjectures. For instance we prove Murre's
  conjectures for complex smooth projective varieties $X$ fibred over
  a surface (via a flat morphism) by quadrics, or by complete
  intersections of dimension $4$ of bidegree $(2,2)$.
\end{abstract}

\section*{Introduction}

Let $X$ be a smooth projective complex variety of dimension $d_X$. We
write $H_i(X)$ for the rational homology group $H_i(X,\Q)$, this group
is isomorphic to $H^i(X,\Q)^\vee$. The group $CH_i(X)$ denotes the
rational Chow group of $i$-cycles on $X$ modulo rational equivalence.
It comes with a cycle class map $cl_i : CH_i(X) \r H_{2i}(X)$.\medskip

This article is concerned with Jacob Murre's Chow-K\"unneth
decomposition problem for smooth projective varieties.
In \cite{Murre1}, Murre conjectured the following. \medskip

(A) $X$ has a Chow-K\"unneth decomposition $\{\pi_0, \ldots,
\pi_{2d}\}$ : There exist mutually orthogonal idempotents $\pi_0,
\ldots, \pi_{2d} \in CH_{d_X}(X \times X)$ adding to the identity such
that $(\pi_i)_*H_*(X)=H_i(X)$ for all $i$.

(B) $\pi_0, \ldots, \pi_{2l-1},\pi_{d+l+1}, \ldots, \pi_{2d}$ act
trivially on $CH_l(X)$ for all $l$.

(C) $F^iCH_l(X) := \ker(\pi_{2l}) \cap \ldots \cap \ker(\pi_{2l+i-1})$
doesn't depend on the choice of the $\pi_j$'s. Here the $\pi_j$'s are
acting on $CH_l(X)$.

(D) $F^1CH_l(X) = CH_l(X)_\hom$. \medskip

A variety $X$ that satisfies conjectures (A), (B) and (D) is said to
have a \emph{Murre decomposition}. If moreover the Chow-K\"unneth
decomposition of conjecture (A) can be chosen so that $\pi_i =
{}^t\pi_{2d-i} \in CH_{d_X}(X \times X)$, then $X$ is said to have a
\emph{self-dual Murre decomposition}. The relevance of Murre's
conjectures were demonstrated by Jannsen who proved \cite{Jannsen}
that these are true for all smooth projective varieties if and only if
Bloch and Beilinson's conjecture is true for all smooth projective
varieties. \medskip

Here we are mainly interested in families of quadric hypersurfaces,
although some of the results can be stated in more generality. Our
strategy for constructing Chow-K\"unneth projectors consists in first
computing the Chow groups of the total space $X$.  In \cite{Vial5}, we
already proved

\begin{theorem2} [Theorem 3.4 in  \cite{Vial5}]
  Let $f : X \r B$ be a complex projective dominant morphism onto a
  complex quasi-projective variety $B$ of dimension $d_B$. Assume that
  there is an integer $l$ such that $CH_i(X_b)=\Q$ for all $i\leq l$
  and all closed points $b \in B$.  Then $CH_i(X)$ has niveau $\leq
  d_B$, i.e. it is supported in dimension $i+d_B$, for all $i \leq l$.
\end{theorem2}

Examples for which the theorem above applies are given by varieties
fibred by complete intersections of very low degree. For instance, if
$Q$ is a quadric hypersurface, then we know that $CH_i(Q)=\Q$ for all
$i < \frac{\dim Q}{2}$.  The above theorem then makes it possible to
establish some of the conjectures on algebraic cycles for smooth
projective varieties fibred by quadrics:

\begin{theorem2} [Theorem 4.2 in \cite{Vial5}]
  Let $X$ be a smooth projective complex variety fibred by quadric
  hypersurfaces over a smooth projective variety $B$.  Then \medskip

  $\bullet$ if $\dim B \leq 1$, $X$ is Kimura finite-dimensional
  \cite{Kimura} and satisfies Murre's conjectures;

  $\bullet$ if $\dim B \leq 2$, $X$ satisfies Grothendieck's standard
  conjectures;

  $\bullet$ if $\dim B \leq 3$, $X$ satisfies the Hodge conjecture.
\end{theorem2}

Since smooth projective surfaces have a Murre decomposition
\cite{Murre}, it is natural to seek for a Murre decomposition for
smooth projective varieties fibred by quadrics over a surface. It
turns out that, when $f : X \r B$ is a complex projective flat
morphism from a smooth variety $X$ to a smooth quasi-projective $B$
whose closed fibres are quadrics, it is possible to compute explicitly
most Chow groups of $X$ in terms of the Chow groups of $B$. Precisely,
in section \ref{hypsection}, we prove an analogue of the projective
bundle formula for Chow groups :

\begin{theorem2} \label{projbundle}[Corollary to Theorem \ref{main}]
  Let $f : X \r B$ be a projective flat dominant morphism from a smooth
  quasi-projective complex variety $X$ to a smooth quasi-projective
  complex variety $B$ of dimension $d_B$.  Let $l \geq 0$ be an
  integer.  Assume that
  $$CH_{l-i}(X_b)= \Q$$ for all $0 \leq i \leq \min (l,d_B)$ and for
  all closed points $b$ of $B$. Then $CH_l(X)$ is isomorphic to
  $\bigoplus_{i=0}^{d_X-d_B} CH_{l-i}(B)$ via the action of
  correspondences.
\end{theorem2}

When the closed fibres of $f$ are quadrics, we thus obtain that $CH_l(X)$
is isomorphic to $\bigoplus_{i=0}^{d_X-d_B} CH_{l-i}(B)$ in a strong
sense for all $l \leq \frac{d_X-d_B-1}{2}$. When furthermore $B$ is a
surface, theorem \ref{projbundle} is the prerequisite for constructing
idempotents in $CH_{d_X}(X \times X)$. In section \ref{quadrics}, we
carry out the construction of a Chow-K\"unneth decomposition for $X$
fibred by quadrics over a surface and prove

 \begin{theorem2} [Corollary to Theorem \ref{Murrequadrics}] Let $X$
   be a smooth projective complex variety which is the total space of
   a flat family of quadrics over a smooth projective curve or
   surface.  Then $X$ has a self-dual Murre decomposition which
   satisfies the motivic Lefschetz conjecture.
\end{theorem2}

This theorem generalises a previous result of del Angel and
M\"uller-Stach \cite{dAMS} where a Murre decomposition was constructed
for threefolds fibred by conics over a surface.\medskip

The {motivic Lefschetz conjecture} stipulates, for $\{\pi_i, 0 \leq i
\leq 2d_X\}$ a Chow-K\"unneth decomposition for $X$, that the
morphisms of Chow motives $(X,\pi_{2d_X-i}^\hom) \r
(X,\pi_{i}^\hom,d_X-i)$ induced by intersecting $d_X-i$ times with a
hyperplane section are isomorphisms for all $0 \leq i \leq d_X$. The
motivic Lefschetz conjecture follows from a combination of Kimura's
finite-dimensionality conjecture with the Lefschetz standard
conjecture. It should be noted that, in order to prove theorem
\ref{Murrequadrics}, no reference to Kimura's finite-dimensionality
property is used.  Furthermore, it doesn't seem possible to prove
theorem \ref{Murrequadrics} by using the approach of
Gordon-Hanamura-Murre \cite{GHM}, as in \emph{loc. cit.} $f$ is
assumed to be smooth away from a finite number of points and to have a
relative Chow-K\"unneth decomposition. Here we only assume $f$ to be
flat and in the proof of theorem \ref{Murrequadrics} we don't consider
the existence of a relative Chow-K\"unneth decomposition for $f$.

\paragraph{Notations.} Chow groups are always meant with rational
coefficients. The group $CH_i(X)$ is the $\Q$-vectorspace with basis
the $i$-dimensional irreducible subschemes of $X$ modulo rational
equivalence.

In section \ref{quadrics}, motives are defined in a covariant setting
and the notations are those of \cite{VialCK}.  Briefly, a Chow motive
$M$ is a triple $(X,p,n)$ where $X$ is a variety of pure dimension
$d$, $p \in CH_d(X\times X)$ is an idempotent ($p\circ p = p$) and $n$
is an integer. The motive of $X$ is denoted $\h(X)$ and by definition
is the motive $(X,\Delta_X,0)$ where $\Delta_X$ is the class in
$CH_{d_X}(X \times X)$ of the diagonal in $X \times X$.  A morphism
between two motives $(X,p,n)$ and $(Y,q,m)$ is a correspondence in $q
\circ CH_{d+n-m} (X \times Y) \circ p$. If $f : X \r Y$ is a morphism,
$\Gamma_f \in CH_{d}(X \times Y)$ is the class of the graph of $f$. By
definition we have $CH_i(X,p,n) = p_*CH_{i-n}(X)$ and $H_i(X,p,n) =
p_*H_{i-2n}(X)$, where we write $H_i(X) := H^{2d-i}(X(\C),\Q)$ for
singular homology.

\paragraph{Acknowledgements.}  This work is supported by a Thomas
Nevile Research Fellowship at Magdalene College, Cambridge and an
EPSRC Postdoctoral Fellowship under grant EP/H028870/1. I would like
to thank both institutions for their support.

\section{Chow groups for varieties fibred by varieties with Chow
  groups generated by hyperplane sections} \label{hypsection}

We establish a formula that is analogous to the projective bundle
formula for Chow groups.  \medskip

For $X$ a projective variety, $h : CH_l(X) \r CH_{l-1}(X)$ denotes the
intersection with a hyperplane section of $X$. It is well-defined
\cite[Cor. 21.10]{Voisin}. When $X$ is also smooth, consider a smooth
linear hyperplane section $\iota : H \hookrightarrow X$. Let's write
$\Delta_H$ for the diagonal inside $H \times H$. The map $h$ is then
induced by the correspondence $(\iota \times \iota)_*[\Delta_H] \in
CH_{d_X-1}(X \times X)$ that we also denote $h$.

\begin{lemma} \label{dominant} Let $f : X \r B$ be a projective dominant
  morphism between two smooth quasi-projective varieties of respective
  dimension $d_X$ and $d_B$. Then there exists a non-zero integer $n$
  such that $f_*h^{d_X-d_B}f^* : CH_0(B) \r CH_0(B)$ is multiplication
  by $n$. If moreover $B$ is projective, then
  $$\Gamma_f \circ h^{d_X -d_B} \circ {}^t\Gamma_f = n \cdot \Delta_B
  \in CH_{d_B}(B \times B).$$
\end{lemma}
\begin{proof}
  This follows from the projection formula applied to $(f \circ
  \iota)_*(f \circ \iota)^*$ and, when $B$ is projective, from Manin's
  identity principle. See \cite[Example 1 p. 450]{Manin}.
\end{proof}

\begin{lemma} \label{orth} Let $f : X \r B$ be a projective dominant
  morphism between two smooth quasi-projective varieties. Then
  $f_*h^{d_X-d_B}f^* : CH_0(B) \r CH_0(B)$ is the zero map for all $l
  < d_X - d_B$. If moreover $B$ is projective, then $\Gamma_f \circ
  h^l \circ {}^t\Gamma_f = 0 \in CH_{d_X-l}(B \times B)$ for all $l <
  d_X - d_B$.
 \end{lemma}
 \begin{proof} Let's first assume that $B$ is projective. Let $\iota :
   H^l \hookrightarrow X$ be a smooth linear section of $X$ of
   codimension $l$ that dominates $B$ and let $h^l$ be the class of
   $(\iota \times \iota)( \Delta_{H^l})\in X \times X$.  By definition
   we have $\Gamma_f \circ h^l \circ {}^t\Gamma_f =
   (p_{1,4})_*(p_{1,2}^*{}^t\Gamma_f \cap p_{2,3}^*h^l \cap
   p_{3,4}^*\Gamma_f)$, where $p_{i,j}$ denotes projection from $B
   \times X \times X \times B$ to the $(i,j)$-th factor. These
   projections are flat morphisms, therefore by flat pullback we have
   $p_{1,2}^*{}^t\Gamma_f = [{}^t\Gamma_f \times X \times B]$,
   $p_{2,3}^*h^l = [B \times \Delta_{H^l} \times B]$ and
   $p_{3,4}^*\Gamma_f = [B \times X \times \Gamma_f]$. It is easy to
   see that the closed subschemes ${}^t\Gamma_f \times X \times B$, $B
   \times \Delta_{H^l} \times B$ and $B \times X \times \Gamma_f$ of
   $B \times X \times X \times B$ intersect properly.  Their
   intersection is given by $\{(f(h),h,h,f(h)) : h \in H\} \subset B
   \times X \times X \times B$. Since $f$ is projective, this is a
   closed subset of dimension $d_X-l$ and its image under the
   projection $p_{1,4}$ has dimension $d_B$, which is strictly less
   than $d_X -l$ by the assumption made on $l$. The projection
   $p_{1,4}$ is a proper map and hence by proper pushforward we get
   that $(p_{1,4})_* [\{(f(h),h,h,f(h)) \in B \times X \times X \times
   B : h \in H^l \}] =0$.

   When $B$ is only assumed to be quasi-projective, the arguments
   above can be adapted by using refined intersections as in
   \cite[Remark 16.1]{Fulton} and by noticing that the graph
   $\Gamma_f$ seen as a subscheme of $X \times B$ is proper over $X$
   and over $B$ and that $H \times H$ is proper over $X$ via the two
   projections.
 \end{proof}

 \begin{proposition} \label{inj} Let $f : X \r B$ be a projective
   dominant morphism from a smooth quasi-projective variety $X$ of
   dimension $d_X$ to a smooth quasi-projective variety $B$ of
   dimension $d_B$.  Then the map
  $$(*) \ \ \ \ \ \bigoplus_{i=0}^{d_X-d_B} h^{d_X -d_B -i} \circ f^* \ : \
  \bigoplus_{i=0}^{d_X-d_B} CH_{l-i}(B) \longrightarrow CH_l(X)$$ is
  injective.
\end{proposition}

\begin{proof} Thanks to lemma \ref{dominant} and to lemma \ref{orth},
  we have that $$f_* \circ h^i \circ f^* : CH_l(B) \r CH_{l+d_X-d_B -
    i}(B)$$ is multiplication by a non-zero integer if $i=d_X-d_B$ and
  is zero if $i<d_X - d_B$.

  Consider the map $ \bigoplus_{j=0}^{d_X-d_B} f_* \circ h^j : CH_l(X)
  \rightarrow \bigoplus_{j=0}^{d_X-d_B} CH_{l-j}(B)$.  In order to
  prove the injectivity of $ \bigoplus_{i=0}^{d_X-d_B} h^{d_X -d_B -i}
  \circ f^*$, it suffices to show that the composite
  $$\Big( \bigoplus_{j=0}^{d_X-d_B} f_* \circ h^j\Big) \ \circ \
  \Big(\bigoplus_{i=0}^{d_X-d_B} h^{d_X -d_B -i} \circ f^*\Big) \ : \
  \bigoplus_{i=0}^{d_X-d_B} CH_{l-i}(B) \longrightarrow CH_l(X)
  \longrightarrow \bigoplus_{j=0}^{d_X-d_B} CH_{l-j}(B)$$ is an
  isomorphism. Indeed it follows from lemma \ref{dominant} and from
  lemma \ref{orth} that this composite map can be represented by an
  upper triangular matrix whose diagonal entries' action on
  $CH_{l-i}(B)$ is given by multiplication by $n$ for some $n \neq 0$.
\end{proof}

\begin{proposition} Let $f : X \r B$ be a flat dominant morphism from
  a quasi-projective variety $X$ of dimension $d_X$ to a
  quasi-projective variety $B$ of dimension $d_B$.  Let $l \geq 0$ be
  an integer.  Assume that $$CH_{l-i}(X_{\eta_{B_i}})= \Q$$ for all $0
  \leq i \leq \min (l,d_B)$ and for all closed irreducible subschemes
  $B_i$ of $B$ of dimension $i$, where $\eta_{B_i}$ is the generic
  point of $B_i$.

  Then the map
  $$(*) \ \ \ \ \ \bigoplus_{i=0}^{d_X-d_B} h^{d_X -d_B -i} \circ f^* \ : \
  \bigoplus_{i=0}^{d_X-d_B} CH_{l-i}(B) \longrightarrow CH_l(X)$$ is
  surjective.
\end{proposition}
\begin{proof}
  The case when $d_B = 0$ is obvious. Let's proceed by induction on
  $d_B$. We have the localization exact sequence $$\bigoplus_{D \in
    B^{1}} CH_l(X_D) \longrightarrow CH_l(X) \longrightarrow
  CH_{l-d_B}(X_{\eta_B}) \longrightarrow 0,$$ where the direct sum is
  taken over all irreducible divisors of $B$. If $l \geq d_B$, let $Y$
  be a closed subscheme of $X$ obtained as the scheme-theoretic
  intersection of $d_X - l$ hyperplanes in general position. Then, by
  Bertini, for a suitable choice of hyperplanes, $Y$ is irreducible,
  has dimension $l$ and is such that $f|_Y : Y \r B$ is dominant. The
  restriction map $CH_l(X) \rightarrow CH_{l-d_B}(X_{\eta_B})$ is by
  definition the direct limit of the flat pullback maps $CH_l(X) \r
  CH_l(X_U)$ taken over all open subsets $U$ of $B$. Therefore
  $CH_l(X) \rightarrow CH_{l-d_B}(X_{\eta_B})$ sends the class of $Y$
  to the class of $Y_{\eta_B}$ inside $CH_{l-d_B}(X_{\eta_B})$. But
  then this class is non-zero because $Y_{\eta_B}$ is irreducible.
  Moreover, if $[B]$ denotes the class of $B$ in $CH_{d_B}(B)$, then
  the class of $Y$ is equal to $h^{d_X- l } \circ f^* [B]$ in
  $CH_l(X)$.  Therefore, the composite map
  $$CH_{d_B}(B) \stackrel{h^{d_X-l} \circ f^*}{\longrightarrow} CH_l(X)
  \r CH_{l-d_B}(X_{\eta_B})$$ is surjective.

  Consider now the fibre square \begin{center} $ \xymatrix{ X_D
      \ar[d]_{f_D} \ar[r]^{j_D'} & X \ar[d]^{f} \\ D \ar[r]^{j_D} &
      B.}$
  \end{center} Then $f_D : X_D \r D$ is flat and its fibres above
  points of $D$ satisfy the assumptions of the theorem. Therefore, by
  the inductive assumption, we have a surjective map
  $$\bigoplus_{i=0}^{d_X-d_B} h^{d_X -d_B -i} \circ f_D^* \ : \
  \bigoplus_{i=0}^{d_X-d_B} CH_{l-i}(D) \longrightarrow CH_l(X_D).$$
  Furthermore, since $f$ is flat and $j_D$ is proper, we have the
  formula \cite[1.7]{Fulton} $$j_{D*}' \circ h^{d_X -d_B -i} \circ
  f_D^* = h^{d_X -d_B -i} \circ f^* \circ j_{D*} \ : \ CH_{l-i}(D) \r
  CH_l(X).$$ Therefore, the image of $(*)$ contains the image of
  $$\bigoplus_{D \in B^1} \bigoplus_{i=0}^{d_X-d_B} j'_{D*} \circ
  h^{d_X -d_B -i} \circ f_D^* \ : \ \bigoplus_{i=0}^{d_X-d_B}
  CH_{l-i}(D) \longrightarrow CH_l(X).$$ Altogether, this implies that
  the map $(*)$ is surjective.
\end{proof}

We can now gather the statements and proofs of the two previous
propositions into the following.

\begin{theorem} \label{main} Let $f : X \r B$ be a flat projective
  dominant morphism from a smooth quasi-projective variety $X$ of
  dimension $d_X$ to a smooth quasi-projective variety $B$ of
  dimension $d_B$.  Let $l \geq 0$ be an integer.  Assume that
  $$CH_{l-i}(X_{\eta_{B_i}})= \Q$$ for all $0 \leq i \leq \min
  (l,d_B)$ and for all closed irreducible subschemes $B_i$ of $B$ of
  dimension $i$, where $\eta_{B_i}$ is the generic point of $B_i$.

  Then the map
  $$(*) \ \ \ \ \ \bigoplus_{i=0}^{d_X-d_B} h^{d_X -d_B -i} \circ f^* \ : \
  \bigoplus_{i=0}^{d_X-d_B} CH_{l-i}(B) \longrightarrow CH_l(X)$$ is
  an isomorphism. Moreover the map $$ \bigoplus_{i=0}^{d_X-d_B} f_*
  \circ h^i \ : \ CH_l(X) \longrightarrow \bigoplus_{i=0}^{d_X-d_B}
  CH_{l-i}(B)$$ is also an isomorphism. \qed
\end{theorem}

\begin{proposition} \label{Chowgroups} Let $f : X \r B$ be a morphism
  of complex varieties with $B$ irreducible and let $F$ be the
  geometric generic fibre of $f$. Then there is a subset $U \subseteq
  B(\C)$ which is a countable intersection of nonempty Zariski open
  subsets such that for each point $p \in U$, $CH_i(X_p)$ is
  isomorphic to $CH_i(F)$ for all $i$.
\end{proposition}
\begin{proof} Cf. \cite[Proposition 3.2]{Vial5}.
\end{proof}

We then have the following corollaries to theorem \ref{main}.

\begin{corollary} \label{corollary-main} Let $f : X \r B$ be a
  projective flat dominant morphism from a smooth quasi-projective
  complex variety $X$ to a smooth quasi-projective complex variety $B$
  of dimension $d_B$.  Let $l \geq 0$ be an integer.  Assume that
  $$CH_{l-i}(X_b)= \Q$$ for all $0 \leq i \leq \min (l,d_B)$ and for
  all closed points $b$ of $B$.

  Then the conclusion of theorem \ref{main} holds.
\end{corollary}
\begin{proof} Let $B_i$ be an irreducible closed subscheme of $B$ of
  dimension $i$ and let $f|_{B_i} : X|_{B_i} \r B_i$ be the
  restriction of $f$ to $B_i$.  If $CH_{l-i}(X_b)= \Q$ for all closed
  points $b \in B$, then proposition \ref{Chowgroups} applied to
  $f|_{B_i}$ implies that $CH_{l-i}(X_{\overline{\eta}_{B_i}})= \Q$.
  Here $\overline{\eta}_{B_i}$ denotes a geometric generic point of
  $B_i$. But then it is well-known \cite[Ex. 1.7.6]{Fulton} that for a
  scheme $X$ over a field $k$, the pull-back map $CH_*(X) \r
  CH_*(X_{\overline{k}})$ is injective. We are thus reduced to the
  statement of theorem \ref{main}.
\end{proof}

\begin{corollary} \label{corollary-quadric} Let $f : X \r B$ be a flat
  dominant morphism from a smooth projective complex variety $X$ to a
  smooth projective complex variety $B$ of dimension $d_B$ whose
  closed fibres are quadric hypersurfaces.  Then the conclusion of
  theorem \ref{main} holds for any $l \leq \lfloor \frac{d_X-d_B-1}{2}
  \rfloor$.
\end{corollary}
\begin{proof} It is well-known (see e.g. \cite{ELV}) that, for a
  quadric hypersurface $Q$, $CH_i(Q)=\Q$ for all $i \leq \lfloor
  \frac{\dim Q -1}{2} \rfloor$. Corollary \ref{corollary-main} thus
  applies.
\end{proof}

\section{Murre's conjectures for total spaces of flat families of
  quadric hypersurfaces over a surface}
\label{quadrics}

The main result of this section is the following.

\begin{theorem} \label{Murrequadrics} Let $f : X \r S$ be a flat
  dominant morphism from a smooth projective complex variety $X$ to a
  smooth projective complex surface $S$ whose closed fibres $X_s$
  satisfy $CH_l(X_s)=\Q$ for all $l \leq \frac{d_X-3}{2}$. Then $X$
  has a self-dual Murre decomposition and $X$ satisfies the motivic
  Lefschetz conjecture.
\end{theorem}

\begin{corollary}
  Let $f : X \r S$ be a flat dominant morphism from a smooth
  projective complex variety $X$ to a smooth projective complex
  surface $S$ whose closed fibres are either quadric hypersurfaces or
  complete intersections of dimension $4$ and bidegree $(2,2)$. Then
  $X$ has a self-dual Murre decomposition and $X$ satisfies the
  motivic Lefschetz conjecture.
\end{corollary}
\begin{proof}
  The Chow groups of a quadric hypersurface $Q$ satisfy $CH_i(Q)=\Q$
  for all $i \leq \frac{\dim Q -1}{2}$ and the Chow groups of a
  complete intersection $X_{2,2}$ of dimension $4$ and bidegree
  $(2,2)$ satisfy $CH_0(X_{2,2}) = CH_1(X_{2,2}) = \Q$. This is for
  example proved in \cite{ELV}.
\end{proof}

Before we proceed to a proof of theorem \ref{Murrequadrics}, we
consider the case when $f : X \r S$ is a smooth quadric fibration.

\subsection{The case of smooth families}

In this subsection we are given $f : X \r B$ a smooth surjective
morphism between smooth projective varieties with fibres being quadric
hypersurfaces. In this case there are several ways to compute the Chow
motive of $X$ in terms of the Chow motive of $B$. Since we are going
to prove a more general statement we only give some indication on
proofs.

\paragraph{Smooth families with a relative Chow-K\"unneth
  decomposition.} Recall that quadric hypersurfaces are cellular
varieties and that smooth quadric hypersurfaces are homogeneous
varieties. Assume first that $X$ has the structure of a relative
cellular variety over $B$. Then K\"ock \cite{Kock} proved that $X$ has
a relative Chow-K\"unneth decomposition over $B$ in the sense of
\cite{GHM}.  If $f$ is only assumed to be smooth, then Iyer
\cite{Iyer2} showed that $f$ is \'etale locally trivial and deduced
that $f$ has a relative Chow-K\"unneth decomposition. By using the
technique of Gordon-Hanamura-Murre \cite{GHM}, it is then possible to
prove that
\begin{center} $\h(X) \simeq \bigoplus_{l=0}^{d_X-d_B} \h(B)(l)$ for
  $d_X-d_B$ odd, and \medskip

  $\h(X) \simeq \bigoplus_{l=0}^{d_X-d_B} \h(B)(l) \oplus
  \h(B)(\frac{d_X-d_B}{2})$ for $d_X-d_B$ even.
\end{center} Actually, in the case when $X$ has the structure of a
relative cellular variety over $B$, this follows immediately from
Manin's identity principle.

When $d_X-d_B$ is odd we develop below an approach that bypasses the
use of the fact that the smooth family $f : X \r B$ is \'etale locally
trivial. This approach is the starting point towards extending the
above result for smooth families to flat families.

\paragraph{Smooth families of odd relative dimension.} If $Q$ is a
smooth projective odd-dimensional quadric, we have that $CH_l(Q) = \Q$
for all $0 \leq l \leq \dim Q$ so that $Q$ has the same Chow groups
(with rational coefficients) as the projective space of dimension
$\dim Q$. Thus when $f$ has odd relative dimension, the situation is
very similar to the case of projective bundles:
Corollary \ref{corollary-main} gives isomorphisms for all $0
\leq l \leq d_X$
$$\bigoplus_{i=0}^{d_X-d_B} h^{d_X -d_B -i} \circ f^* \ : \
\bigoplus_{i=0}^{d_X-d_B} CH_{l-i}(B) \longrightarrow CH_l(X).$$ If $H
\hookrightarrow X$ is a smooth hyperplane section of $X$, then for $Y$
a smooth projective variety $H \times Y \hookrightarrow X \times Y$ is
also a smooth hyperplane section of $X \times Y$. Therefore the
isomorphism above induces a similar isomorphism for the smooth map $f
\times \id_Y : X \times Y \r B \times Y$ and Manin's identity
principle applies to give an isomorphism of Chow motives $$\h(X)
\simeq \bigoplus_{l=0}^{d_X - d_B} \h(B)(l).$$

There is yet another way of proceeding and this will be the path we
will follow to prove the case of flat families over a surface (in
which case the arguments above do not suffice as the Chow groups of
singular quadrics are not all equal to $\Q$). We only
give a sketch and point out where the difficulty is.  Thanks to lemma
\ref{dominant} we can define idempotents $\pi_0, \ldots, \pi_{d_X -
  d_B} \in CH_{d_X}(X \times X)$ as $$\pi_l := \frac{1}{n} \cdot
h^{d_X - d_B -l} \circ {}^t\Gamma_f \circ \Gamma_f \circ h^l.$$
Lemma \ref{dominant} shows that these idempotents satisfy $(X,\pi_l)
\simeq \h(B)(l)$ and lemma \ref{orth} shows that $\pi_l \circ \pi_{l'}
= 0$ for all $l' > l$. Moreover theorem \ref{main} shows that $CH_i(X)
= \sum_l (\pi_l)_*CH_i(X)$ for all $i$.

Now we have the following non-commutative Gram-Schmidt process
\cite[lemma 2.12]{Vial3}.

\begin{lemma} \label{linalg} Let $V$ be a $\Q$-algebra and let $k$ be
  a positive integer. Let $\pi_0, \ldots, \pi_n$ be idempotents in $V$
  such that $\pi_i \circ \pi_j = 0$ whenever $i -j < k$ and $i \neq
  j$. Then the endomorphisms $$p_i := (1-\frac{1}{2}\pi_n) \circ
  \cdots \circ (1-\frac{1}{2}\pi_{i+1}) \circ \pi_i \circ
  (1-\frac{1}{2}\pi_{i-1}) \circ \cdots \circ (1-\frac{1}{2}\pi_0)$$
  define idempotents such that $p_i \circ p_j = 0$ whenever $i -j <
  k+1$ and $i \neq j$.
\end{lemma}

\begin{proposition} \label{GS} Let $X$ be a smooth projective variety
  of dimension $d$.  Let $\pi_0, \ldots, \pi_s \in CH_d(X \times X)$
  be idempotents such that $\pi_l \circ \pi_{l'} = 0$ for all $l' >
  l$, then the non-commutative Gram-Schmidt process of lemma
  \ref{linalg} gives mutually orthogonal idempotents $\{p_l\}_{l \in
    \{0,\ldots,s\}}$ such that we have isomorphisms of Chow motives
  $(X,\pi_l) \simeq (X,p_l)$ for all $l$.
\end{proposition}
\begin{proof}
  In order to produce mutually orthogonal idempotents, it is enough to
  apply lemma \ref{linalg} $(l-1)$-times. It is then enough to check
  the isomorphisms of Chow motives after each application of the
  process of lemma \ref{linalg}. Such isomorphisms are simply given by
  the correspondences $p_l \circ \pi_l$ ; the inverse of $p_l \circ
  \pi_l$ is $\pi_l \circ p_l$ as can be readily checked.
\end{proof}

This way we get mutually orthogonal idempotents $p_0, \ldots, p_{d_X -
  d_B}$ such that $(X,p_l) \simeq \h(B)(l)$. In order to conclude it
would be nice to know that $CH_*(X) = CH_*(X,\sum p_l)$.  In that case
$(X,\Delta_X - \sum p_l)$ is a Chow motive with trivial Chow groups
which implies that $\Delta_X = \sum p_l$ and hence $\h(X) \simeq
\bigoplus_{l=0}^{d_X-d_B} \h(B)(l)$.  However, it is not clear how to
prove from here that $CH_*(X) = CH_*(X,\sum p_l)$ and this explains
why the proof of the next section might seem convoluted.

\subsection{Proof of theorem \ref{Murrequadrics}}

We now assume that $f : X \r S$ is a flat dominant morphism defined
over $\C$ from a smooth projective variety $X$ to a smooth projective
surface $S$ whose closed fibres $X_s$ satisfy $CH_l(X_s)=\Q$ for all
$l \leq \frac{d_X-3}{2}$.
\medskip

The general strategy for proving theorem \ref{Murrequadrics} consists
in exhibiting some idempotents modulo rational equivalence with a
prescribed action on Chow groups or cohomology groups and then to turn
them into an orthogonal family. We do this step by step. Here's a
rough outline. At each step we check that the idempotents form an
ordered family which is ``semi-orthogonal'' in the sense that $P_i
\circ P_j = 0$ for $j>i$.  This makes it possible to run the
non-commutative Gram-Schmidt process of lemma \ref{linalg} to get an
orthogonal family of idempotents. Such an orthonormalising process
does not affect the action of the idempotents on cohomology.  At each
step we need to keep track of the action of the idempotents on the
Chow groups of $X$.  For this purpose we check at each step that $P_j
\circ P_i$ acts trivially on $CH_l(X)$ for all $l$ and all $j > i$.

First we construct idempotents $\pi_{2i}^{tr}$ that factor through
surfaces and ``not through curves''.

Then for $l \leq \lfloor \frac{d_X - 3}{2} \rfloor$ we construct
idempotents $p_{2l}^{alg}$ and $p_{2l+1}$. We check that those
idempotents act the way we want them to on the Chow groups of $X$. We
deduce that they act as wanted on the cohomology of $X$.

We then define $p_{d_X-1}^{alg}$ if $d_X$ is odd and mutually
orthogonal idempotents $p_{d_X-2}^{alg}$ and $p_{d_X-1}$ if $d_X$ is
even. We use a different construction than the one before as here we
check directly that they act the way we want on the cohomology of $X$.

Finally we define $p_{2l} := p_{2l}^{alg} + p_{2l}^{tr}$ for $2l <
d_X$, and $p_l := {}^tp_{2d_X-l}$ for $l>d_X$.  \medskip

\noindent \emph{Step 1.} Let $ \pi^{tr,S}_2 \in CH_2(S \times S)$ be
an idempotent with the following properties. Its homology class is the
orthogonal projector on the orthogonal complement of $H_{1,1}(X) \cap
H_2(X,\Q)$ inside $H_2(X,\Q)$ with respect to the choice of a
polarisation on $S$. It acts trivially on $CH_1(S)$ and on $CH_2(S)$
and $$(\pi^{tr,S}_2)_* CH_0(S) = \ker (\mathrm{alb}_S : CH_0(S)_\hom
\r \mathrm{Alb}_S(k)).$$ Such an idempotent exists, see
\cite{KMP}.\medskip

\noindent \emph{Step 2.}  From lemma \ref{dominant}, let $n$ be the
non-zero integer such that $\Gamma_f \circ h^{d_X - d_S} \circ
{}^t\Gamma_f = n \cdot \Delta_S \in CH_{2}(S \times S).$ We set for $
2i \neq d_X$, $$\pi^{tr}_{2i} := \frac{1}{n} \cdot h^{d_X - d_S -i+1}
\circ {}^t\Gamma_f \circ \pi_2^{tr,S} \circ \Gamma_f \circ h^{i-1} \in
CH_{d_X}(X \times X).$$ It is understood that $h^l=0$ for $l < 0$.
Because the correspondence $h$ is self-dual (that is $h={}^th$) we see
that $$\pi^{tr}_{2d_X - 2i} = {}^t\pi^{tr}_{2i}.$$ It is expected that
the correspondence $\pi^{tr}_{2i}$ induces the projector on the
orthogonal complement of $H_{i,i}(X) \cap H_{2i}(X,\Q)$ inside
$H_{2i}(X,\Q)$. This will become apparent at the end of step 6.
\medskip

\noindent \emph{Step 3. Orthogonality relations among the
  $\pi^{tr}_{2i}$.}

\begin{proposition} \label{semiorthogonality} The $\pi^{tr}_{2i}$'s
  satisfy the following identities: \medskip

    $\bullet$ $\pi^{tr}_{2i} \circ \pi^{tr}_{2i} = \pi^{tr}_{2i}$ for
    $2i \neq d_X$,

    $\bullet$ $\pi^{tr}_{2i} \circ \pi^{tr}_{2j} = 0$ for all $i<j$
    with $2i,2j \neq d_X$.
  \end{proposition}
\begin{proof}
  By definition of the $\pi^{tr}_{2i}$'s we have $$\pi^{tr}_{2i} \circ
  \pi^{tr}_{2j} = \frac{1}{n^2} \cdot h^{d_X - d_S -i+1} \circ
  {}^t\Gamma_f \circ \pi_2^{tr,S} \circ \Gamma_f \circ h^{d_X - d_S
    +i-j} \circ {}^t\Gamma_f \circ \pi_{2}^{tr,S} \circ \Gamma_f \circ
  h^{j-1}.$$ If $i=j$, then lemma \ref{dominant} gives $\pi^{tr}_{2i}
  \circ \pi^{tr}_{2i} = \pi^{tr}_{2i}$. If $i<j$, then lemma
  \ref{orth} ensures that $\Gamma_f \circ h^{d_X - d_S +i-j} \circ
  {}^t\Gamma_f=0$ so that $\pi^{tr}_{2i} \circ \pi^{tr}_{2j} = 0$.
 \end{proof}

Because we will need to keep track of the action of the idempotents on
the Chow groups of $X$ after orthonormalising the family
$\{\pi^{tr}_{2i} : 2i \neq d_X\}$, we state the following.

\begin{proposition} \label{trivialaction} The correspondence
  $\pi^{tr}_{2j} \circ \pi^{tr}_{2i}$ acts trivially on $CH_*(X)$ for
  all $i \neq j$.
\end{proposition}
\begin{proof} The correspondence $\pi^{tr}_{2i}$ factors through
  $\pi_2^{tr,S}$ and hence, thanks to step 1, $\pi^{tr}_{2i}$ acts
  trivially on $CH_j(X)$ for $j \neq i-1$.
\end{proof}

\noindent \emph{Step 4. Orthonormalising the $\pi^{tr}_{2i}$.}  By
proposition \ref{GS}, after having applied lemma \ref{linalg} a finite
number of times to the set of idempotents $\{\pi_{2i}^{tr} : 2i \neq
d_X\}$, we get a set of mutually orthogonal idempotents $\{p_{2i}^{tr}
: 2i \neq d_X\}$ such that the Chow motives $(X,\pi_{2i}^{tr})$ and
$(X,p_{2i}^{tr})$ are isomorphic for all $i$ with $2i \neq d_X$.
\begin{proposition} \label{action1} Let $2i \neq d_X$.  The action
  of $p^{tr}_{2i}$ on $CH_l(X)$ coincides with the action of
  $\pi^{tr}_{2i}$ for all $l$.
\end{proposition}
\begin{proof}
  Considering the formula of lemma \ref{linalg} that defines the
  idempotents $p^{tr}_{2i}$ inductively from the the idempotents
  $\pi^{tr}_{2i}$, this follows from proposition \ref{trivialaction}.
\end{proof}

\noindent \emph{Step 5.} Let's define the following idempotent $$Q :=
\Delta_X - \sum_{2i \neq d_X} p_{2i}^{tr} \in CH_{d_X}(X
\times X).$$

\begin{definition}
  Let $(X,P)$ be a Chow motive. The subgroup of $CH_i(X,P)$ consisting
  of algebraically trivial cycles is denoted $CH_i(X,P)_\alg$. This
  subgroup can be shown to coincide with the image of the map $P_* :
  CH_i(X)_\alg \r CH_i(X)_\alg$. It is said to be \emph{representable}
  if there exist a curve $C$ and a correspondence $\alpha \in
  \Hom(\h_1(C)(i),(X,P))$ such that the induced map $\alpha_* :
  CH_0(C)_\alg \r CH_i(X,P)_\alg$ is surjective.
\end{definition}

\begin{proposition}
  The group $CH_l(X,Q)_\alg$ of $l$-cycles modulo rational equivalence
  which are algebraically equivalent to zero is representable for all
  $l \leq \lfloor \frac{d_X - 3}{2} \rfloor$.
\end{proposition}

\begin{proof}
  By proposition \ref{action1}, the action of $Q$ on $CH_l(X)$
  coincides with the action of $Q' := \Delta_X -
  \sum_{2i \neq d_X} \pi_{2i}^{tr}$. Consider the map
  $$\Phi := \bigoplus_{i=l-2}^{l} h^{d_X -d_S -i} \circ f^* \circ
  (\Delta_S - \pi_2^{tr,S})_* \ : \ \bigoplus_{i=l-2}^{l} CH_{l-i}(S)
  \longrightarrow CH_l(X) .$$ By corollary \ref{corollary-main}, the
  map $\Psi := \bigoplus_{i=l-2}^{l} f_* \circ h^i : CH_l(X) \r
  \bigoplus_{i=l-2}^{l} CH_{l-i}(S)$ is an isomorphism. Moreover, we
  have $Q' = \Phi \circ \Psi$ so that $\im(\Phi) = (Q')_*CH_l(X)$. We
  can then conclude that $Q_*CH_l(X)_\alg$ is representable because
  $(\Delta_S - \pi_2^{tr,S})_* CH_k(S)_\alg$ is representable for all
  $k$.
\end{proof}

\noindent \emph{Step 6. The idempotents $p_{2i}^{alg}$ and
  $p_{2i+1}$.}  We first construct idempotents $p_{2l}^{alg}$ and
$p_{2l+1}$ for $l \leq \lfloor \frac{d_X - 3}{2} \rfloor$ that act
appropriately on Chow groups. Then we construct idempotents
$p_{d_X-2}^{alg}$ and $p_{d_X-1}$ if $d_X$ even and an idempotent
$p_{d_X-1}^{alg}$ if $d_X$ odd that act appropriately on homology.
\medskip

In order to define  $p_{2l}^{alg}$ and $p_{2l+1}$ for $l \leq
\lfloor \frac{d_X - 3}{2} \rfloor$ we use the construction of \cite[\S
1]{Vial3}. Let's recall it. By Jannsen's theorem \cite{Jannsen3}, the
category of motives for numerical equivalence is abelian semi-simple.
Therefore we can construct idempotents modulo numerical equivalence
$\overline{p}_{2l}^{alg}$ and $\overline{p}_{2l+1}$ such that
\begin{center}
  $(X, \overline{p}_{2l}^{alg}) = \sum \im (\overline{\mathds{1}}(l)
  \r (X, \overline{Q} ))$ and $(X, \overline{p}_{2l+1}) = \sum \im
  (\overline{\h}_1(C)(l) \r (X, \overline{Q} ))$.
\end{center}
Here the first sum runs over all morphisms $\overline{\mathds{1}}(l)
\r (X, \overline{Q} )$ and the second sum runs over all curves $C$ and
all morphisms $\overline{\h}_1(C)(l) \r (X, \overline{Q} )$. We then
see that there is an integer $n$ such that $(X,
\overline{p}_{2l}^{alg})$ is isomorphic to
$\overline{\mathds{1}}(l)^{\oplus n}$ and a curve $C$ such that $(X,
\overline{p}_{2l+1})$ is isomorphic to a direct summand of
$\overline{\h}_1(C)(l)$. Because $\End(\overline{\mathds{1}}^{\oplus
  n}) = \End(\mathds{1}^{\oplus n})$ and $\End(\overline{\h}_1(C)) =
\End({\h}_1(C))$, we can lift the idempotents
$\overline{p}_{2l}^{alg}$ and $\overline{p}_{2l+1}$ to idempotents
$p_{2l}^{alg}$ and $p_{2l+1}$ modulo rational equivalence which are
orthogonal to $\Delta_X - Q$ and such that $(X,p_{2l}^{alg})$ is
isomorphic to $\mathds{1}(l)^{\oplus n}$ and $(X,p_{2l+1})$ is
isomorphic to a direct summand of $\h_1(C)(l)$. If we construct these
idempotents one after the other and replace $Q$ by $Q$ minus the last
constructed idempotent at each step, we see that in addition to being
orthogonal to $\Delta_X - Q$, the idempotents \{$p_{2l}^{alg},p_{2l+1}
: l \leq \lfloor \frac{d_X - 3}{2} \rfloor \}$ can be constructed
so as to form a family of mutually orthogonal idempotents.\medskip

Now we check that these idempotents act the way we want on the Chow
groups of $X$.

\begin{lemma} [lemma 3.3. in \cite{Vial1}] \label{curve} Let $P \in
  CH_{d_X}(X \times X)$ be an idempotent and assume that
  $CH_0(X,P)_\alg$ is representable. Assume also
  that for all curves $C$ and all correspondences $\alpha \in
  \Hom(\h_1(C),(X,P))$ we have that $\alpha$ is numerically trivial.
  Then $CH_0(X,P)=0$.
\end{lemma}
\begin{proof}
  Let $C$ be a curve and let $\gamma \in CH_1(C \times X)$ be a
  correspondence such that $\gamma_*CH_0(C)_\alg = P_* CH_0(X)_\alg $.
  Let then $\alpha := P \circ \gamma \circ \pi_1^C \in
  \Hom(\h_1(C),(X,P))$. By \cite[Th. 3.6]{Vial4}, which follows a
  decomposition of the diagonal argument \`a la Bloch-Srinivas
  \cite{BS}, we get that $P = P_1 + P_2$ where $P_2$ is supported on
  $D \times X$ for some divisor $D$ in $X$ and $P_1 = \alpha \circ
  \beta$ for some $\beta \in \Hom((X,P),\h_1(C))$.  By Chow's moving
  lemma $P_2$ acts trivially on $CH_0(X)$ so that $P_1 = \alpha \circ
  \beta$ acts as the identity on $P_*CH_0(X)$. By assumption,
  $\overline{\alpha} = 0$ and thus $\overline{\beta} \circ
  \overline{\alpha} = 0$. Because $\End(\h_1(C)) =
  \End(\overline{\h}_1(C))$, we get that $\beta \circ \alpha = 0$. It
  follows that $P_*CH_0(X) = 0$.
\end{proof}

\begin{lemma} \label{point}
  Let $P \in CH_{d_X}(X \times X)$ be an idempotent and assume that
  $CH_0(X,P)$ is a finite-dimensional $\Q$-vector space. Assume also
  that for all correspondences $\alpha \in \Hom(\mathds{1},(X,P))$ we
  have that $\alpha$ is numerically trivial.  Then $CH_0(X,P)=0$.
\end{lemma}
\begin{proof}
  The lemma can be proved along the same lines as lemma \ref{curve}.
\end{proof}

From lemmas \ref{curve} and \ref{point}, we get

\begin{proposition} \label{induction2} Let $P \in CH_{d_X}(X \times
  X)$ be an idempotent and assume that $CH_0(X,P)_\alg$ is
  representable. Assume that $(X,\overline{P})$ has no direct summand
  isomorphic to $\overline{\mathds{1}}$ or to a direct summand of the
  $\overline{\h}_1$ of a curve.
  Then $CH_0(X,P)=0$. \qed
\end{proposition}

The next lemma was mentioned to me by Bruno Kahn.

\begin{lemma} \label{induction} Let $P \in CH_{d_X}(X \times X)$ be an
  idempotent and assume that $CH_0(X,P) = 0$. Then there exists a
  smooth projective variety $Y$ of dimension $d_X-1$ and an idempotent
  $P' \in CH_{d_X-1}(Y \times Y)$ such that $(X,P) \simeq (Y,P',1)$.
\end{lemma}
\begin{proof}
  For a proof, see \cite[Theorem 2.1]{VialCK}.
\end{proof}

Let $$P := \Delta_X - \sum (p_{2l}^{alg} + p_{2l+1} + p_{2l+2}^{tr})$$
where the sum is taken over all $l \leq \lfloor \frac{d_X - 3}{2}
\rfloor$.  We are finally in a position to prove the crucial

\begin{proposition} \label{actionChow} For all $l \leq \lfloor
  \frac{d_X - 3}{2} \rfloor$ we have $CH_l(X) = (p_{2l}^{alg} +
  p_{2l+1} + p_{2l+2}^{tr})_*CH_l(X)$.
\end{proposition}
\begin{proof}
  The idempotent $\pi_{2l+2}^{tr}$ acts trivially on $CH_{l'}(X)$ for
  all $l' \neq l$ and so does $p_{2l+2}^{tr}$ by proposition
  \ref{action1} (or more simply because $p_{2l+2}^{tr}$ factors
  through $\pi_{2l+2}^{tr}$ by the formula of lemma \ref{linalg}). By
  construction, the idempotents $ p_{2l}^{alg}$ and $p_{2l+1}$ also
  act trivially on $CH_{l'}(X)$ for all $l' \neq l$.  Therefore it
  suffices to prove that $P_*CH_l(X) = 0$ for $l \leq \lfloor
  \frac{d_X - 3}{2} \rfloor$. The case $l=0$ is proposition
  \ref{induction2}. By proposition \ref{induction}, we get that
  $(X,P)$ is isomorphic to $(Y,P',1)$ for some smooth projective $Y$
  and some idempotent $P' \in CH_{\dim Y}(Y \times Y)$. We can then
  apply proposition \ref{induction2} to $(Y,P')$ and we obtain
  $CH_1(X,P) = 0$. An easy induction concludes the proof.
\end{proof}

Proposition \ref{actionChow} yields that the Chow motive $(X,P)$ has
trivial Chow groups in degrees less than $\lfloor \frac{d_X - 3}{2}
\rfloor$.

It follows from lemma \ref{induction} that there exist a smooth
projective variety $Y$ and an idempotent $q \in CH_{\dim Y}(Y \times
Y)$ such that $(X,P)$ is isomorphic to $(Y,q,\lfloor \frac{d_X -
  1}{2} \rfloor)$. Let $\alpha \in \Hom ((X,P), (Y,q,\lfloor
\frac{d_X - 1}{2} \rfloor))$ denote such an isomorphism and let
$\beta$ be its inverse.  In \cite[\S3]{VialCK}, 
orthogonal idempotents $q_0$ and $q_1 \in \End((Y,q))$ with the
following properties are constructed: \medskip

$\bullet$  $(q_0)_*H_*(Y) = q_*H_0(Y)$ and $(q_1)_*H_*(Y) =
q_*H_1(Y)$.

$\bullet$ The Chow motive $(Y,q_0)$ is isomorphic to a direct sum of
Chow motives of points.

$\bullet$ The Chow motive $(Y,q_1)$ is isomorphic to a direct summand
of the Chow motive of a curve. \medskip

\noindent Let's then define the idempotent $p_{d_X-1}^{alg} := \beta
\circ q_0 \circ \alpha$ if $d_X$ is odd and mutually orthogonal
idempotents $p_{d_X-2}^{alg} := \beta \circ q_0 \circ \alpha$ and
$p_{d_X-1}:= \beta \circ q_1 \circ \alpha$ if $d_X$ is even.  \medskip

By construction the idempotents $p_{2l}^{alg}$ and $p_{2l}^{tr}$ are
mutually orthogonal for all $0 < l < \frac{dim X}{2}$. Let's thus
define the idempotent $$p_{2l} := p_{2l}^{alg} + p_{2l}^{tr}.$$
We also set $p_0 = p_0^{alg}$. We have thus now at our disposal a
set $\{p_l\}_{0 \leq l < d_X}$ of mutually orthogonal idempotents.
Modulo homological equivalence, these define the K\"unneth projectors:

\begin{proposition} \label{Kunneth}
  The mutually orthogonal idempotents $\{p_l\}_{0 \leq l < d_X}$
  satisfy $$(p_l)_*H_*(X) = H_l(X).$$
\end{proposition}
\begin{proof}
  For weight reasons we immediately see that $(p_l)_*H_*(X) \subseteq
  H_l(X)$ for all $l<d_X$. By proposition \ref{actionChow}, we have
  that $CH_l(X,\Delta_X - \sum_{l'<d_X} p_{l'})=0$ for $l \leq \lfloor
  \frac{d_X - 3}{2} \rfloor$. As in the discussion above, lemma
  \ref{induction} then shows that there exists $Y$ and an idempotent
  $q \in CH_{\dim Y}(Y \times Y)$ such that $(X,\Delta_X -
  \sum_{l'<d_X} p_{l'})$ is isomorphic to $(Y,q,\lfloor \frac{d_X -
    1}{2} \rfloor)$. Clearly $H_l(Y,q,\lfloor \frac{d_X - 1}{2}
  \rfloor) = 0$ for $l < 2\lfloor \frac{d_X - 1}{2} \rfloor$ so that
  $(p_l)_*H_*(X) = H_l(X)$ for $l < 2\lfloor \frac{d_X - 1}{2}
  \rfloor$. It then follows from the definitions of $p_{d_X-1}^{alg}$
  when $d_X$ is odd and of $p_{d_X-2}^{alg}$ and $p_{d_X-1}$ when
  $d_X$ is even that $(p_l)_*H_*(X) = H_l(X)$ for the remaining $l$'s
  that is for $2\lfloor \frac{d_X - 1}{2}\rfloor \leq l < d_X$.
\end{proof}

By Poincar\'e duality we then have for $l < d_X$ $$({}^tp_l)_*H_*(X) =
H_{2d_X-l}(X).$$ We are thus led to set for $l > d_X$ $$p_l := {}^t
p_{2d_X - l}.$$ \medskip

\noindent \emph{Step 7. More orthogonality relations.}

\begin{lemma} \label{1way} Let $V$ and $W$ be two smooth projective
  varieties and let $\gamma \in CH^0(V \times W)$ be a correspondence
  such that $\gamma_*$ acts trivially on zero-cycles.  Then $\gamma =
  0$.
\end{lemma}
\begin{proof} We can assume that $V$ and $W$ are both connected.  The
  cycle $\gamma$ is equal to $a\cdot [V \times W]$ for some $a \in
  \Q$. Let $z$ be a zero-cycle on $V$. Then $\gamma_*z = a \cdot \deg
  z \cdot [W]$. This immediately implies $a=0$.
\end{proof}

The following lemma will be used in the proof of lemma \ref{lefrel}.

\begin{lemma} \label{2way} Let $\gamma \in CH^1(V \times W)$ be a
  correspondence such that both $\gamma_*$ and $\gamma^*$ act
  trivially on zero-cycles. Then $\gamma = 0$.
\end{lemma}
\begin{proof} We can assume $V$ and $W$ are connected. We have
  $\mathrm{Pic}(V \times W) = \mathrm{Pic}(V) \times [W] \oplus [V]
  \times \mathrm{Pic}( W)$.  The cycle $\gamma$ is thus equal to $D_1
  \times [W] \oplus [V] \times D_2$ for some divisors $D_1 \in
  CH^1(V)$ and $D_2 \in CH^1(W)$. Let $z$ be a zero-cycle on $V$.
  Then $\gamma_*z = \deg z \cdot D_2$.  This immediately implies
  $D_2=0$. Likewise, if $z \in CH_0(W)$, $\gamma^* z =0$ implies
  $D_1=0$. We have thus proved that $\gamma=0$.
\end{proof}

\begin{proposition} \label{vanishing} Let $C$ and $C'$ be smooth
  projective curves and let $S$ be a smooth projective surface
  together with an idempotent $\pi_2^{tr,S}$ as in Step 1. Then
  \medskip

 $\bullet$ $\Hom (\h_1(C)(l),\h_1(C'))= 0$ for $l>0$,

 $\bullet$ $\Hom ((S,\pi_2^{tr,S},l),\h_1(C))=0$ for $l>0$.
\end{proposition}
\begin{proof}
  The result is trivial for dimension reasons if $l>1$. Let's thus
  consider the case $l=1$. If $\gamma$ is a morphism that belongs
  to $\Hom (\h_1(C)(1),\h_1(C'))$ (resp. $\Hom
  ((S,\pi_2^{tr,S},1),\h_1(C))$), then $\gamma$ is an element of
  $CH^0(C \times C')$ (resp. $CH^0(S \times C)$) such that $\gamma_*$
  acts trivially on zero-cycles. By lemma \ref{1way} we get $\gamma =
  0$.
\end{proof}

\begin{proposition} \label{semiorthogonality2} Let $\{p_i : i \neq
  d_X\}$ be the idempotents constructed in Step 6. For all $0 \leq i,j
  < d_X$ we have the following relations.  \medskip

$\bullet$ $p_{i} \circ p_j = 0$ for $i \neq j$,

$\bullet$ $p_{i} \circ {}^tp_j = 0$.
\end{proposition}
\begin{proof} The first point is clear by construction of the $p_i$'s
  for $0 \leq i < d_X$. Concerning the second point, we already know
  from Step 4 that $p_{2i}^{tr} \circ {}^tp_{2j}^{tr} = 0$. Here is
  what is left to prove. \medskip

  $\bullet$ $p_{2i}^{alg} \circ {}^tp_{j} = 0$ for $0 \leq 2i,j < \dim
  X$. This follows immediately for dimension reasons and from the fact
  that $p_{2i}^{alg} $ factors through a zero-dimensional variety.

  $\bullet$ $p_{2i+1} \circ {}^tp_{2j+1} = 0$ for $0 \leq 2i+1,2j+1 <
  d_X$. The correspondence $p_{2i+1} \circ {}^tp_{2j+1}$ factors
  through a correspondence $\gamma \in \Hom (\h_1(C_i)(d_X
  -2i-1),\h_1(C_i))$ for some curve $C_i$. By proposition
  \ref{vanishing}, the group $\Hom (\h_1(C_i)(\dim X
  -i-j-1),\h_1(C_i))$ is zero for $d_X -i- j -1 > 0$ and hence
  $p_{2i+1} \circ {}^tp_{2j+1} = 0$.

  $\bullet$ $p_{2i+1} \circ {}^tp_{2j}^{tr} = 0$ for $0 \leq 2i+1,2j <
  d_X$. The correspondence $p_{2i+1} \circ {}^tp_{2j}^{tr}$ factors
  through a correspondence $\gamma \in \Hom ((S,\pi_2^{tr,S},d_X
  -i-j-1),\h_1(C_i))$ for some curve $C_i$. By proposition
  \ref{vanishing}, the group $\Hom ((S,\pi_2^{tr,S},d_X
  -i-j-1),\h_1(C_i))$ is zero for $d_X - i - j - 1 > 0$ and hence
  $p_{2i+1} \circ {}^tp_{2j}^{tr} = 0$.
\end{proof}

 \noindent \emph{Step 8.  Orthonormalising the $p_i$'s.} By
 proposition \ref{semiorthogonality2}, the set of idempotents $\{p_l :
 l \neq d_X\}$ is such that $p_l \circ p_{l'} = 0$ for $l<l'$ and
 $l,l' \neq d_X$. Therefore, we can apply proposition \ref{GS} to get
 a new set of mutually orthogonal idempotents, that we denote $\{\Pi_l
 : l \neq d_X\}$. We then set $$\Pi_{d_X} := \Delta_X - \sum_{l \neq
   d_X} \Pi_l.$$

\noindent \emph{Step 9. The $\Pi_i$'s define a self-dual
  Chow-K\"unneth decomposition for $X$.} We are now in a position to
state the following.

\begin{proposition} \label{CK}
  The set $\{\Pi_l : 0 \leq l \leq 2d_X\}$ defines a Chow-K\"unneth
  decomposition for $X$ that enjoys the following properties:
  \medskip

  $\bullet$ Self-duality, i.e. $\Pi_l = {}^t\Pi_{2d_X-l}$ for all $l$.

  $\bullet$ $(X,\Pi_{2l},-l+1)$ is isomorphic to a direct summand of
  the motive of a surface for $2l \neq d_X$.

  $\bullet$ $(X,\Pi_{2l+1},-l)$ is isomorphic to a direct summand of
  the motive of a curve for $2l+1 \neq d_X$.
\end{proposition}
\begin{proof}
  It is easy to see from the fact that $p_l = {}^tp_{2d_X -l}$ for
  all $l \neq d_X$ and from the formula of lemma \ref{linalg} that
  $\Pi_l = {}^t\Pi_{2d_X-l}$ for all $l$. That the decomposition
  $\{\Pi_l : 0 \leq l \leq 2d_X\}$ does indeed induce a K\"unneth
  decomposition, i.e. that $(\Pi_l)_* H_*(X) = H_l(X)$, follows for $l
  \neq d_X$ from the isomorphisms $(X,p_l) \simeq (X,\Pi_l)$ of
  proposition \ref{GS} and from proposition \ref{Kunneth}. It is
  then obvious that $(\Pi_{d_X})_*H_*(X) = H_{d_X}(X)$.

  Concerning the last two points, this follows again from the fact
  that $(X,\Pi_l)$ is isomorphic to $(X,p_l)$ for all $l\neq d_X$ by
  proposition \ref{GS}, and from the construction of $p_l$ carried out
  in Steps 4 and 6.
\end{proof} \medskip

\noindent \emph{Step 10. On the middle idempotent $\Pi_{d_X}$.} Here
we characterise the support of the idempotent $\Pi_{d_X}$. It is an
essential step towards proving Murre's conjectures for $X$.  Let's
start by showing that the $\Pi_i$'s act the same way as the $p_i$'s on
Chow groups for $i \neq d_X$, i.e. we show that the action on Chow
groups is not altered by the non-commutative Gram-Schmidt process. For
this purpose we need the following.

\begin{proposition} \label{trivialaction2}
  The correspondence ${}^tp_j \circ p_i$ acts trivially on $CH_*(X)$
  for all $i,j < d_X$.
\end{proposition}
\begin{proof} The idempotents $p_{2i}^{alg}$ (resp. $p_{2i}^{tr}$,
  $p_{2i+1}$) factor through $\h(P_i)(i)$ (resp.
  $(S,\pi_2^{tr,S},i-1)$, $\h_1(C_i)(i)$) for some variety $P_i$
  (resp. $S$, $C_i$) of dimension $0$ (resp. $2$, $1$). For dimension
  reasons we thus actually have ${}^tp_j \circ p_i = 0$ for $|2d_X
  - i - j| > 3$. By construction, we also have ${}^t p_{2j}^{tr} \circ
  p_{2i}^{tr} = 0$. Here are the remaining cases. \medskip
 
  $\bullet$ ${}^tp_{d_X - 1} \circ p_{d_X - 1}$ acts trivially on
  $CH_*(X)$. Indeed, when $d_X$ is even, then ${}^tp_{d_X - 1} \circ
  p_{d_X - 1}$ factors through a morphism $\gamma \in
  \Hom(\h_1(C),\h_1(C)(1))$ that clearly acts trivially on
  $CH_*(\h_1(C))$. When $d_X$ is odd, there are two cases that need be
  treated. First ${}^tp_{d_X - 1}^{alg} \circ p_{d_X - 1}^{alg} = 0$
  because it factors through a morphism $\gamma \in
  \Hom(\h(P),\h(P)(1))$ for some zero-dimensional variety $P$.
  Secondly, ${}^tp_{d_X - 1}^{alg} \circ p_{d_X - 1}^{tr}$ acts
  trivially on $CH_*(X)$ because it factors through a morphism $\gamma
  \in \Hom((S,\pi_2^{tr,S}),\h(P)(2))$ for some zero-dimensional
  variety $P$ and hence $\gamma$ is seen to act trivially on
  $CH_*(S,\pi^{tr,S}_2)$.

  $\bullet$ ${}^tp_{d_X - 2} \circ p_{d_X - 1}$ acts trivially
  on $CH_*(X)$. If $d_X$ is even, then ${}^tp_{d_X - 2} \circ
  p_{d_X - 1}$ factors through a morphism $\gamma \in
  \Hom(\h_1(C),(S,\pi_2^S,1))$ that clearly acts trivially on
  $CH_*(\h_1(C))$. If $d_X$ is odd, then ${}^tp_{d_X - 2} \circ
  p_{d_X - 1}$ factors through a morphism $\gamma \in
  \Hom((S,\pi_2^S),\h_1(C)(2))$ that clearly acts trivially on
  $CH_*(S,\pi_2^S)$.

  $\bullet$ ${}^tp_{d_X - 1} \circ p_{d_X - 2}$ acts trivially
  on $CH_*(X)$. The proof is similar to the previous case and is left
  to the reader.
\end{proof}

\begin{proposition} \label{trivialaction3} The correspondence $p_j
  \circ p_i$ acts trivially on $CH_*(X)$ for all $i \neq j $ with $i,j
  \neq d_X$.
\end{proposition}
\begin{proof}
  Recall that for $i \neq d_X$, we have $p_i = {}^tp_{2d_X-i}$. Then
  the proposition follows from a combination of propositions
  \ref{semiorthogonality2} and \ref{trivialaction2}.
\end{proof}

\begin{proposition} \label{Caction2} Let $i \neq d_X$.  The action
  of $\Pi_{i}$ on $CH_l(X)$ coincides with the action of $p_{i}$ for
  all $l$, i.e. for all $x \in CH_l(X)$ we have $(\Pi_{i})_*x =
  (p_i)_*x$.
\end{proposition}
\begin{proof}
  This can be read off the formula of lemma \ref{linalg} using
  proposition \ref{trivialaction3}.
\end{proof}

\begin{proposition} \label{actionChow2} For all $l \leq \lfloor
  \frac{d_X - 3}{2} \rfloor$ we have $CH_l(X) = (\Pi_{2l} +
  \Pi_{2l+1} + \Pi_{2l+2})_*CH_l(X)$.
\end{proposition}
\begin{proof}
  This follows immediately from propositions \ref{actionChow} and
  \ref{Caction2}.
\end{proof}

We now give the two main propositions concerning the middle
Chow-K\"unneth idempotent $\Pi_{d_X}$.

\begin{proposition} \label{oddmiddle} If $X$ is odd-dimensional, then
  there is a curve $C$ such that $(X,\Pi_{d_X})$ is isomorphic to a
  direct summand of $\h(C)(\frac{d_X-1}{2})$.
\end{proposition}
\begin{proof}
  By proposition \ref{actionChow2} we have $CH_l(X,\Pi_{d_X}) = 0$ for
  all $l \leq \frac{d_X-3}{2}$. Applying $\frac{d_X-1}{2}$ times lemma
  \ref{induction}, we get a smooth projective variety $Z$ of dimension
  $\frac{d_X+1}{2}$ and an idempotent $q$ such that $(X,\Pi_{d_X})
  \simeq (Z,q, \frac{d_X-1}{2})$. By proposition \ref{CK}, we have
  $\Pi_{d_X} = {}^t\Pi_{d_X}$. Therefore, by duality, we get
  $(X,\Pi_{d_X}) \simeq (Z,{}^tq)$. Thus $CH_l(Z,{}^t q) = 0$ for all
  $l \leq \frac{d_X-3}{2}$.  Applying $\frac{d_X-1}{2}$ times lemma
  \ref{induction} to $(Z,{}^t q)$, we get a curve $C$ such that
  $(Z,{}^t q)$ is isomorphic to a direct summand of
  $\h(C)(\frac{d_X-1}{2})$. Dualizing, we see that $(Z,q)$ is
  isomorphic to a direct summand of $\h(C)$. This finishes the proof.
\end{proof}

\begin{proposition} \label{evenmiddle} If $X$ is even-dimensional,
  then there is a surface $S$ such that $(X,\Pi_{d_X})$ is isomorphic
  to a direct summand of $\h(S)(\frac{d_X-2}{2})$.
\end{proposition}
\begin{proof}
  The proof follows the exact same pattern as the proof of the
  proposition \ref{oddmiddle}.
\end{proof} \medskip

\noindent \emph{Step 11. The motivic Lefschetz conjecture for $X$.}

\begin{proposition} \label{transiso} The morphisms $\pi_{2i}^{tr}
  \circ h^{d-2i} \circ {}^t \pi_{2i}^{tr} \in
  \Hom((X,{}^t\pi_{2i}^{tr}),(X,\pi_{2i}^{tr},d-2i))$ are isomorphisms
  for $2i < d$.
\end{proposition}
\begin{proof} We claim that $$\frac{1}{n} \cdot {}^t\pi_{2i}^{tr}
  \circ h^{i-1} \circ {}^t\Gamma_f \circ \Gamma_f \circ h^{i-1} \circ
  \pi_{2i}^{tr}$$ is the inverse of $\pi_{2i}^{tr} \circ h^{d-2i}
  \circ {}^t \pi_{2i}^{tr}$.  Indeed this follows from the formula
  defining the idempotents $\pi_{2i}^{tr}$ and from lemma
  \ref{dominant}.
\end{proof}

By proposition \ref{GS} the motives $(X,p_{2i})$ and $(X,\Pi_{2i})$
are isomorphic for all $2i \neq d_X$. Let's thus consider the
orthogonal decomposition $\Pi_{2i} = \Pi_{2i}^{alg} +\Pi_{2i}^{tr}$
arising from the latter isomorphism and from the decomposition $p_{2i}
= p_{2i}^{alg} +p_{2i}^{tr}$.

\begin{proposition} \label{motlef1} The morphisms $$\Pi_{2i}^{alg}
  \circ h^{d-2i} \circ {}^t \Pi_{2i}^{alg} \in
  \Hom((X,{}^t\Pi_{2i}^{alg}),(X,\Pi_{2i}^{alg},d-2i))$$ for $2i<d_X$
  and $$\Pi_{2i+1} \circ h^{d-2i-1} \circ {}^t \Pi_{2i+1} \in
  \Hom((X,{}^t\Pi_{2i+1}),(X,\Pi_{2i+1},d-2i-1))$$ for $2i+1 < d_X$
  are isomorphisms of Chow motives.
\end{proposition}
\begin{proof}
  By the hard Lefschetz theorem $ h^{d-i}_* : H^{i}(X) \r H_{i}(X)$ is
  an isomorphism for all $i \leq d$. In particular, $ h^{d-i}_* :
  H_*(X,{}^t \Pi_{i}) \r H_{*}(X,\Pi_{i})$ is an isomorphism for all
  $i \leq d$. The isomorphism of proposition \ref{transiso} together
  with the hard Lefschetz theorem in degree $2i$ implies that $
  h^{d-2i}_* : H_*(X,{}^t \Pi_{2i}^{alg}) \r H_{*}(X,\Pi_{2i}^{alg})$
  is an isomorphism. We thus see that the morphisms of the proposition
  induce isomorphisms on homology. We can now conclude with
  \cite[Propositions 5.1 \& 5.2]{VialCK} by saying that
  $(X,\Pi_{2i}^{alg},-i)$ is isomorphic to the motive of a
  zero-dimensional variety and that $(X,\Pi_{2i+1},-i)$ is isomorphic
  to a direct summand of the $\h_1$ of a curve.
 \end{proof}

\begin{lemma} \label{lefrel}
  Let $\alpha \in CH_{2i}(X \times X)$. Then \medskip

  $\bullet$ $\pi_{2j}^{tr} \circ \alpha \circ {}^t\pi_{2i}^{tr} = 0$
  for $j < i$.

  $\bullet$ $\pi_{2j+1} \circ \alpha \circ {}^t\pi_{2i}^{tr} = 0$ for
  $j < i$.

  $\bullet$ $\pi_{2j}^{alg} \circ \alpha \circ {}^t\pi_{2i}^{tr} = 0$
  for $j \leq i$.
\end{lemma}
\begin{proof}
  In the first case, $\pi_{2j}^{tr} \circ \alpha \circ
  {}^t\pi_{2i}^{tr}$ factors through a correspondence $\gamma \in
  CH_{2+i-j}(S \times S)$ such that $\gamma_*z = \gamma^*z = 0$ for
  all $z \in CH_0(S)$. If $i>j+2$ then clearly $\gamma=0$. If $i=j+2$,
  lemma \ref{1way} gives $\gamma=0$. If $i=j+1$, then lemma \ref{2way}
  gives $\gamma = 0$.

  In the second case, there is a curve $C$ such that $\pi_{2j+1} \circ
  \alpha \circ {}^t\pi_{2i}^{tr}$ factors through a correspondence
  $\gamma \in CH_{1+i-j}(S \times C)$ such that $\gamma_*z = 0$ for
  all $z \in CH_0(S)$ and $\gamma^*z' = 0$ for all $z' \in CH_0(C)$.
  This implies $\gamma = 0$ by lemmas \ref{1way} and \ref{2way}.

  Finally, in the last case, there exists a zero-dimensional $P$ such
  that $\pi_{2j}^{alg} \circ \alpha \circ {}^t\pi_{2i}^{tr}$ factors
  through a correspondence $\gamma \in CH_{1+i-j}(S \times P)$ such
  that $\gamma_*z = 0$ for all $z \in CH_0(S)$ and $\gamma^*z' = 0$
  for all $z' \in CH_0(P)$. We conclude as in the previous cases.
\end{proof}

By proposition \ref{motlef1}, in order to prove the motivic Lefschetz
conjecture for $X$, it is enough to show that $\Pi_{2i}^{alg} \circ
h^{d-2i} \circ {}^t \Pi_{2i}^{tr} = 0$ and that $\Pi_{2i}^{tr} \circ
h^{d-2i} \circ {}^t \Pi_{2i}^{tr}$ is an isomorphism. The first point
follows immediately from lemma \ref{lefrel} and from the formula of
lemma \ref{linalg} defining the $\Pi_i$'s in terms of the $p_i$'s.
Concerning the second point, we know by proposition \ref{transiso}
that $\Pi_{2i}^{tr} \circ \pi_{2i}^{tr} \circ h^{d-2i} \circ {}^t
\pi_{2i}^{tr } \circ {}^t\Pi_{2i}^{tr}$ is an isomorphism with inverse
$\frac{1}{n}{}^t \Pi_{2i}^{tr} \circ \pi_{2i}^{tr} \circ h^{i-1} \circ
{}^t\Gamma_f \circ \Gamma_f \circ h^{i-1} \circ \pi_{2i}^{tr} \circ
\Pi_{2i}^{tr}$. We can therefore conclude that $\Pi_{2i}^{tr} \circ
h^{d-2i} \circ {}^t \Pi_{2i}^{tr}$ is an isomorphism if we can show
the equality $$ \Pi_{2i}^{tr} \circ h^{d-2i} \circ {}^t \Pi_{2i}^{tr}
= \Pi_{2i}^{tr} \circ \pi_{2i}^{tr} \circ h^{d-2i} \circ {}^t
\pi_{2i}^{tr } \circ {}^t\Pi_{2i}^{tr}.$$ Having a close look at the
non-commutative Gram-Schmidt process of lemma \ref{linalg} we see that
this reduces to the identities proved in lemma \ref{lefrel}.

The motivic Lefschetz conjecture for $X$ is thus established. \qed

\begin{remark}
  The morphism $\Pi_i \circ h^{d-i} \circ {}^t\Pi_i \in \Hom((X,
  {}^t\Pi_i),(X,\Pi_i,d-i))$ is an isomorphism for any choice of a
  polarisation $h$. In order to see this, we only need to check that
  $\pi_{2i}^{tr} \circ (h')^{d-2i} \circ {}^t \pi_{2i}^{tr} \in
  \Hom((X,{}^t\pi_{2i}^{tr}),(X,\pi_{2i}^{tr},d-2i))$ is an
  isomorphism for all polarisations $h'$. This follows from the fact,
  which is analogous to lemma \ref{dominant}, that for any choice of
  polarisations $h_1, \ldots, h_{d_X-d_S}$ there exists a non-zero
  integer $m$ such that $\Gamma_f \circ h_1 \circ \ldots \circ h_{d_X
    -d_S} \circ {}^t\Gamma_f = m \cdot \Delta_S \in CH_{d_S}(S \times
  S).$
\end{remark}
\medskip

\noindent \emph{Step 12. Murre's conjectures for $X$.} Thanks to
propositions \ref{CK}, \ref{oddmiddle} and \ref{evenmiddle}, the
following proposition settles Murre's conjectures (B) and (D) for $X$.

\begin{proposition}
  Let $X$ be a smooth projective variety of dimension $d$. Suppose $X$
  has a Chow-K\"unneth decomposition $\{\Pi_i\}_{0\leq i \leq 2d}$
  such that, for all $i$, \medskip

  $\bullet$ $\Pi_{2i}$ factors through a surface, i.e. there is a
  surface $S_i$ such that $(X,\Pi_{2i})$ is a direct summand of
  $\h(S_i)(i-1)$.

$\bullet$ $\Pi_{2i+1}$ factors through a curve, i.e.  there is a curve
$C_i$ such that $(X,\Pi_{2i+1})$ is a direct summand of
$\h_1(C_i)(i)$. \medskip

\noindent Then homological and algebraic equivalence agree on $X$, $X$
satisfies Murre's conjectures (A), (B) and (D), and the filtration
does not depend on the choice of a Chow-K\"unneth decomposition as
above.
\end{proposition}
\begin{proof} See \cite[Proposition 6.4]{VialCK}.
 \end{proof}

 \begin{remark} Let $X$ be a smooth projective variety defined over a
   subfield $k$ of $\C$. Assume that there is a flat dominant morphism
   $f : X \r S$ to a smooth projective surface $S$ defined over $k$
   such that for all field extensions $K/k$ and all points $\Spec \ K
   \r S$ the fibre $X_{\Spec \ K}$ is a quadric hypersurface. Then the
   conclusion of theorem \ref{Murrequadrics} holds for $X$, i.e. $X$
   has a self-dual Murre decomposition which satisfies the motivic
   Lefschetz conjecture.
 \end{remark}

 \begin{footnotesize}
  \bibliographystyle{plain} 
  \bibliography{bib} \medskip

  \textsc{DPMMS, University of Cambridge, Wilberforce Road, Cambridge,
    CB3 0WB, UK}
  \end{footnotesize}

\textit{e-mail :}  \texttt{c.vial@dpmms.cam.ac.uk}

\end{document}